\documentclass[12pt]{amsart}
\usepackage{amsmath,amssymb,amscd,amsthm,amscd,mathrsfs,stmaryrd,inputenc,enumerate}

\usepackage[english]{babel}
\usepackage{color}%

\newcommand{\p}{{\mathfrak p}}

\newcommand{\m}{{M}}

\newcommand{\sss}{{\rm Spec}}

\newcommand{\tmax}{t{\rm -Max}}

\def\b{\mathfrak b} 
\def\m{\mathfrak m} 
\def\gg{\ms G} 

\newcommand{\ff}{\mathscr F}
\newcommand{\bbf}{{\bf f}}

\newcommand{\ms}{\mathscr}

\newcommand{\f}{\mathfrak}
\newcommand{\Max}{{\rm Max}}
\newcommand{\sMax}{\star-\Max}

\newcommand{\z}{{\ldots}}

\theoremstyle{break}
\newtheorem{thm}{ \textsc{Theorem}}[section]

\newtheorem{cor}[thm]{ \textsc{Corollary}}
\newtheorem{lem}[thm]{ \textsc{Lemma}}
\newtheorem{prop}[thm]{ \textsc{Proposition}}

\theoremstyle{remark}

\newtheorem{remarks}[thm]{Remarks}
\newtheorem{ex}[thm]{Example}
\newtheorem{defn}[thm]{\bf Definition}

\begin{document}
\title[Star-Invertibility and $t$-finite character]
{Star-Invertibility and $t$-finite character in Integral Domains}
\thanks{2000 {\it Mathematics Subject Classification}.
Primary: 13A15,  13F05; Secondary: 13B30, 13G05.
\newline
{\it Key words and phrases} Finite character, invertible ideal,
star-operation.}
\author[C.A. Finocchiaro]{Carmelo Antonio Finocchiaro}
\address{Dipartimento di Matematica\\ Universit\`{a}
degli studi Roma Tre\\ Largo San Leonardo Murialdo 1, 00146 Roma,
Italy} \email{carmelo@mat.uniroma3.it}
\author[G. Picozza]{Giampaolo Picozza}
\address{Laboratoire d'Analyse, Topologie, Probabilités,
Facult\'e des Sciences de Marseille Saint-Jer\^ome 13397 Marseille
Cedex 20, France} \email {giampaolo.picozza@univ-cezanne.fr}
\author[F. Tartarone]{Francesca Tartarone}
\address{Dipartimento di Matematica\\ Universit\`{a}
degli studi Roma Tre\\ Largo San Leonardo Murialdo 1, 00146 Roma,
Italy} \email{tfrance@mat.uniroma3.it}

\date{\today}
\begin{abstract} Let $A$ be an integral domain. We study new
conditions on families of integral ideals of $A$ in order to get
that $A$ is of $t$-finite character (i.e., each nonzero element of
$A$ is contained in finitely many $t$-maximal ideals). We also
investigate  problems connected with   the local invertibility of
ideals.
\end{abstract}

\maketitle

\section*{Introduction}

In their series of papers about flatness of ideals (\cite{GV,
flat3}), S. Glaz and W. Vasconcelos address the problem of when flat
ideals are invertible. Among other results, they conjecture that
(\cite[\S 3]{GV}):

\smallskip
(C1) \lq\lq \textit{In an H-domain (i.e. a domain in which
$t$-maximal ideals are divisorial), faithfully flat ideals are
projective}".

\smallskip
Since, for ideals in a domain, projective is equivalent to
invertible and, by a result of D.D. Anderson and D.F. Anderson
\cite{aa2}, faithfully flat is equivalent to locally principal, the
condition \lq\lq faithfully flat ideals are projective"    can be
restated as ``locally principal ideals are invertible''. This last
condition appears again in a paper by S. Bazzoni \cite{bazzoni}, who
conjectures that:

\smallskip
(C2) \lq\lq \textit{Pr\"ufer domains with this property have the
finite character on maximal ideals}".

\smallskip

In \cite{PT}, the conjectures (C1) and (C2) are put in relation and
it is shown that they are incompatible. Moreover, an example in
which (C1) fails is exhibited. In the same time, several proofs of
Bazzoni's conjecture  appeared: W.C. Holland, J. Martinez, W. Wm. Mc
Govern and M. Tesemma provided a proof of the original conjecture
(\cite{hmmt}), relying strongly on methods of the theory of lattice
ordered groups (their proof has been recently translated in a ring
theoretic language in \cite{mc}); F. Halter-Koch gave a proof in the
more general context of $r$-Pr\"ufer monoids \cite{hk}; M. Zafrullah
\cite{zaf} also studied Bazzoni's conjecture for P$v$MD's but
developing different techniques from F. Halter Koch's paper.

In this paper, we further generalize these results. As pointed out
in \cite{PT}, the finite character on maximal ideal is not necessary
to have that locally principal ideals are invertible. In fact,
Noetherian domains have obviously this property, but they have not
necessarily the finite character. However, Noetherian domains have
the $t$-finite character, which turns out to be a sufficient
condition to obtain the invertibility of locally principal ideals
(on the other hand, Example \ref{inv not tinv} shows  that
$t$-finite character is not necessary, answering a question posed in
\cite{PT}).

So, as in F. Halter-Koch and M. Zafrullah's papers (\cite{zaf, hk}),
we focus on $t$-versions of Bazzoni's conjecture. We study the
following three different conditions, which are all equivalent to
the condition \lq\lq each locally principle ideal is invertible" in
the case of Pr\"ufer domains:

\smallskip

1) $t$-locally $t$-finite ideals are $t$-finite;

2) $t$-locally principal ideals are $t$-invertible;

3) $t$-locally $t$-invertible $t$-ideals are $t$-invertible;

\smallskip
All conditions 1)-3) hold when the domain has the $t$-finite
character. Conversely we prove that condition 1) implies the
$t$-finite character with the fairly mild hypothesis of the
$t$-local $v$-coherence of the domain (Theorem
\ref{thm:characterization t-finite}). 

Finally, with similar techniques, we show that condition 2) implies
the $t$-finite character in domains which are locally GCD
(Proposition \ref{locgcd}) and that condition  3) implies the
$t$-finite character in domains which are locally P$v$MD
(Proposition \ref{locpvmd}).

\bigskip

Briefly, we recall some terminology and definitions that we will
freely use in the   paper.

Let $A$ be an integral domain with quotient field $K$. Recall that a
star operation on $A$ is a map $\star$ from the set $F(A)$  of
nonzero fractional ideals of $A$ in itself, $I \mapsto I^\star$,
verifying the following properties  for all $I,J \in F(A)$, $0 \neq
x \in K$:

\begin{enumerate}[(a)]
\item $A^\star = A$, $(xI)^\star = xI^\star$;

\item $I \subseteq J \Rightarrow I^\star \subseteq J^\star$;

\item $(I^\star)^\star = I^\star$.
\end{enumerate}

A star operation $\star$ on $A$  is said \textit{finite type star
operation}  if for all $I \in F(A)$,
$$I^\star = \bigcup \{J^\star : J \subseteq I, J \mbox{ is finitely generated}
\}.$$

An ideal $I \in F(A)$ such that $I^\star = I$ is called a
$\star-$ideal. As usual,  $I^{-1} := (A:I) = \{x \in K; xI \subseteq
A\}$,   and $I$ is $\star-$invertible if $(II^{-1})^\star = A$. An
ideal $I$ is $\star-$finite if there exists a finitely generated
ideal $J$ such that $J^\star = I^\star$. If $\star$ is of finite
type, $J$ can always be taken inside $I$. A $\star-$invertible ideal
is $\star-$finite.

If $\star$ is a   finite type star operation, the set of
$\star-$ideals has (proper) maximal elements, which are prime ideals
and are called $\star-$maximal ideals (this set is denoted by
$\star-\Max(A)$). Moreover, every integral $\star-$ideal is
contained in a $\star-$maximal ideal. We say that $A$ has the
$\star-$finite character if each $\star-$ideal (equivalently, each
element) of $A$ is contained in   finitely many  $\star-$maximal
ideals.

We will call an integral ideal of $A$ simply {\it ideal}. We will
also denote by $\mathcal I^\star(A)$ (resp. $\mathcal I^\star_{\rm
fin}(A)$) the collection of all $\star-$ideals (resp. $\star-$finite
$\star-$ideals) of $A$.

The identity is a finite type star operation and it is usually
denoted by $d$. Other classical examples of star operations are:

\begin{itemize}
\item the $v$-operation, defined by $I \mapsto I^v = (A:(A:I))=
(I^{-1})^{-1}$;

\item the $t$-operation, defined by $$I \mapsto I^t = \bigcup \{J^v
: J \subseteq I, J \mbox{ is finitely generated} \}.$$ The
$t$-operation is of finite type and it is maximal in a way that if
$\star$ is another finite type star operation on $A$, then $I^\star
\subseteq I^t$ for each $I \in F(A)$. \end{itemize}

\section{Comaximal families of ideals}

If $\ms F$ is a nonempty collection of subsets of a set, we denote
by $\bigcap \ms F$ the intersection of all the members of $\ms F$.
Moreover, if $\ms F$ is a  finite collection of ideals of a ring, we
denote by $\prod \ms F$ the product of all members of $\ms F$.

 If $(X,\leq)$ is a partially ordered set, we shall denote by
$\Max_{\leq}(X)$ the set of all the maximal  elements of $X$.  With
these notation, if $\star$ is of finite type, then we have that
$\star-\Max(A)=\Max_{\subseteq}(\mathcal
I^\star(A)\setminus\{A\})$).

If $E$ is a subset of $A$,  $V(E)$ denotes the set of all the prime
ideals of $A$ containing $E$.

\begin{defn}
Let $A$ be an integral domain and $\star$ a star operation on $A$.
\begin{enumerate}
\item[\rm (i)] { We say that an ideal $\f a$ of $A$ \textit{satisfies the property ($\star$-c.a.)}
 ($\star$-closure under addition) if, for each pair $\f a_1,\f a_2$ of proper $\star-$finite $\star-$ideals of $A$ containing $\f a$, then $(\f a_1+\f a_2)^\star$ is still a proper ideal of $A$.}
\item[\rm (ii)] If $\ms F$ is a collection of ideals of $A$ and $a\in A\setminus \{0\}$, we say that $\ms F$ is a  {\rm a $\star$-comaximal collection over $a$} if $a\in \bigcap \ms F$ and $(\f a_1+\f a_2)^\star =A$, for every $\f a_1,\f a_2\in \ms F$, with $\f a_1\neq \f a_2$.
\end{enumerate}
\end{defn}

\begin{lem}\label{uno}
Let $A$ be an integral domain and $\star$ be a finite type star
operation on $A$. Assume that { $\f a$ is a proper $\star-$finite
$\star-$ideal satisfying the property ($\star$-c.a.)}. Then
$$
\f m_{\f a}:=\{x\in A:((x)+\f a)^\star \subsetneq A\}
$$
is the only $\star-$maximal ideal containing $\f a$.
\end{lem}

\begin{proof}
It is clear that $\f m_{\f a}$ is the union of the set of the
$\star-$maximal ideals containing $\f a$. So, let $x, y \in \f m_{\f
a}$. Since $\f a $ satisfies the $\star$-c.a. property, $(((x) + \f
a) + ((y) + \f a))^\star$ is a proper $\star-$ideal, so it is
contained in a $\star-$maximal ideal $\f m \in V(\f a)$. Thus, $x,y
\in \f m$ and so $x - y \in \f \m \subseteq \f m_{\f a}$. That $ax
\in \f  m_{\f a}$, for all $a \in A$, $x \in \f m_{\f a}$ is
straightforward. So $ \f m_{\f a}$ is an ideal of $A$. Moreover, if
$\f b = (b_1, b_2, \ldots, b_n)$ is a finitely generated ideal
contained in $\f m_{\f a}$, we have that for each $i = 1, 2, \ldots,
n$, $b_i$ is contained in a $\star-$maximal ideal $\f m_i \in V(\f
a)$. So $((b_i) + \f a)^ \star \subsetneq A$. Since $\f a$ satisfies
the property $\star$-c.a., it follows that $(\f b + \f a)^\star
\subsetneq A$. Thus, $\f b \subseteq \f m$ for some $\m \in V(\f
a)$. So we have that $\f b^\star \subseteq \f m^\star = \f m
\subseteq \f m_{\f a}$, and so $\f m_{\f a}$ is a $\star-$ideal
(since $\star$ is of finite type).

So, it follows that $\f m_{\f a}$ is a $\star-$ideal containing all
the $\star-$maximal ideals that contain $\f a$, that is $\f m_{\f
a}$ is a $\star-$maximal ideal and it is the only one that contains
$\f a$.
\end{proof}

{ Let $A$ be an integral domain and $a$ be a nonzero element of $A$.
We denote by $(\Sigma')^\star_a$ the collection of all the
$\star$-comaximal families over $a$ of
 proper $\star-$finite $\star-$ideals of $A$.
We call $\Sigma^\star_a$ the collection of all the $\star$-comaximal
families  over $a$ of
 proper $\star-$finite $\star-$ideals of $A$ satisfying the property ($\star$-c.a.).}\\
 Obviously, we have the containment $\Sigma^\star_a \subseteq (\Sigma')^\star_a$.
\begin{prop}\label{starcaratterefinitezza}
Let $A$ be an integral domain, $\star$ a finite type star operation
on $A$ and $a\in A \backslash \{0\}$. Assume that there exists a
finite collection of ideals $\ms F\in\Sigma^\star_a\cap
\Max_\subseteq((\Sigma ')^\star_a)$. Then $V(a)\cap \sMax(A)=\{\f
m_{\f a}:\f a\in \ms F\}$ ($\f m_{\f a}$ as in Lemma \ref{uno}),
and, in particular, it is finite.
\end{prop}

\begin{proof} Let $\ms F:=\{\f a_1,\z, \f a_n\}$. By Lemma \ref{uno}, { the inclusion} $V(a)\cap \sMax(A)\supseteq\{\f m_{\f a_1},\z,\f
m_{\f a_n}\}$ is immediate. Conversely, assume, by contradiction,
that there exists a $\star-$maximal ideal $\f m$ containing $a$ such
that $\f m\neq \f m_{\f a_1},\z,\f m_{\f a_n}$. In particular, we
have $\f m\nsubseteq \f m_{\f a_i}$, for each $i\in\{1,\z,n\}$, and
thus, { by the Prime Avoidance Lemma}, we can pick an element ${
x}\in \f m$ such that $((x)+\f a_i)^\star=A$, for each $i=1,\z, n$.
Now, set $\f a:=(x,a)^\star$. It is clear that $\f a$ is a proper
$\star-$finite $\star-$ideal of $A$ (since $\f a\subseteq \f m$).
For every $i=1,\z, n$, we have $A=((x)+\f a_i)^\star\subseteq (\f
a+\f a_i)^\star$, and hence $\widehat{\ms F}:=\ms F\cup\{\f a\}$ is
a $\star$-comaximal collection over $a$ of proper $\star-$finite
$\star-$ideals of $A$. Moreover $\f a\neq \f a_i$, for each $i=1,\z,
n$, since $x\in \f a\setminus \f m_{\f a_i}$ and $\f a_i\subseteq \f
m_{\f a_i}$. Thus we have $\ms F\subsetneq \widehat{\ms F}$, a
contradiction, by virtue of the maximality of $\ms F$ {  in $(\Sigma
')^\star_a$}.
\end{proof}

The following remark allows us to get some sufficient condition
about the existence of a collection of ideals as stated in
Proposition \ref{starcaratterefinitezza}.
\begin{prop}\label{condizionesufficientemax}
Let $A$ be an integral domain, $\star$ a star operation on $A$ and
$a\in A \backslash \{0\}$. Assume that, for every ideal $\f a\in
\mathcal I^\star_{\rm fin}(A)$ containing $a$, there exists an ideal
$\f b\in \mathcal I^\star_{\rm fin}(A)$ { with the property
($\star$-c.a.)} containing $\f a$. Then, the following statements
hold.

\begin{enumerate}[\rm(i)]
\item The family $\Sigma^\star_a$, partially ordered by the inclusion $\subseteq$, has maximal
elements.
\item Every maximal element of $\Sigma^\star_a$ is a maximal element of
$(\Sigma')^\star_a$ (partially ordered by the same inclusion
$\subseteq$).

\item There exists a maximal element of $((\Sigma')^\star_a, \subseteq)$ which is an element of $\Sigma^\star_a$.
\end{enumerate}
\end{prop}

\begin{proof} (i) By assumption, there exists a proper $\star-$finite
$\star-$ideal $\f b$ { with the property ($\star$-c.a.)} containing
$a$ . Then $\{\f b\}\in \Sigma^\star_a$. Now, let $\{\ms
G_\lambda:\lambda\in \Lambda\}\subseteq \Sigma^\star_a$ be a chain
and set $\ms G:=\bigcup_{\lambda \in \Lambda}\ms G_{\lambda}$.
Clearly, { every ideal in $\ms G$ has the property ($\star$-c.a.)}.
Now, let $\f a_1,\f a_2\in \ms G$. Then, pick elements
$\lambda_1,\lambda_2\in\Lambda$ such that $\f a_i\in \ms
G_{\lambda_i}$ ($i=1,2$). We can assume, without loss of generality,
that $\ms G_{\lambda_1}\subseteq \ms G_{\lambda_2}$. Then, since
$\ms G_{\lambda_2}$ is a $\star$-comaximal collection over $a$, we
have $(\f a_1+\f a_2)^\star=A$. Hence $\ms G\in \Sigma^\star_a$ and
thus it is an upper bound of the chain $\{\ms
G_\lambda:\lambda\in\Lambda\}$. Now, statement (i) follows
immediately by Zorn's Lemma.

(ii) Let $\ms G$ be a maximal element of $\Sigma^\star_a$. Assume,
by contradiction, that there exists a collection $\ms H \in
(\Sigma')^\star_a$
 such that $\ms G\subsetneq \ms H$, and let $\f a\in \ms H\setminus
\ms G$. Since $\ms G\cup\{\f a\}\notin\Sigma^\star_a$ is clearly a
$\star$-comaximal collection over $a$ of proper $\star-$finite
$\star-$ideals, {\rm it follows that the ideal $\f a$ has not the
property ($\star$-c.a.)}. By assumption, { we can pick an ideal $\f
b\in \mathcal I^\star_{\rm fin}(A)$ containing $\f a$ such that $\f
b$ satisfies the property ($\star$-c.a.)}. Then $\f b \notin \ms G$,
otherwise $\f b \in \ms H$, but $\f b$ is not $\star$-comaximal with
$\f a$, and so this would be a contradiction. Thus, $\ms G\subsetneq
\widehat{\ms G}:=\ms G\cup\{\f b\}$. By construction, { each ideal
of $\widehat{\ms G}$ satisfies the property ($\star$-c.a.)}.
Moreover, every element of $\widehat{\ms G}$ is different by $\f a$,
and hence, since $\ms H$ is a $\star$-comaximal collection over $
a$, we have $A=(\f c+\f a)^\star\subseteq (\f c+\f b)^\star$. This
shows that $\widehat{\ms G}\in \Sigma^\star_a$, a contradiction, by
the maximality of $\ms G$. Thus (ii) is proved. Statements (iii)
follows immediately by (i) and (ii). \end{proof}
\begin{prop}\label{fin_imply_add}
Let $A$ be an integral domain, $\star $ a star operation on $A$ and
$a\in A \backslash \{0\}$. Assume that every family $\gg \in
(\Sigma')^\star_a$ is finite. Then, { for each $\f a\in \mathcal
I^\star_{\rm fin}(A)$, with $a\in \f a$, there exists an ideal  $\f
b\in \mathcal I^\star_{\rm fin}(A)$ containing $\f a$ such that $\f
b$ satisfies the condition ($\star$-c.a.).}
\end{prop}

\begin{proof} Assume, by a way of contradiction, that the statement
is false. Then, in particular, { $\f a$ does not satisfy the
property ($\star$-c.a.)}. Hence, there exist proper ideals $\f
a_1^1,\f a_2^1\in \mathcal I^\star_{\rm fin}(A)$ containing $\f a$
such that $(\f a_1^1+\f a_2^1)^\star =A$. For each $i\geq 2$, by an
easy induction argument, we can pick { proper $\star-$finite
$\star-$ideals
 $\f a_1^i,\f a_2^i$ containing $\f a_1^{i-1}$} such that $(\f a_1^i+\f a_2^i)^\star=A$. Now, fix $i\geq 2 $ and $j<i$. Keeping in mind
that $$\f a_2^j\supseteq \f a_1^{j-1}\supseteq \f a_1^i,$$ it
follows immediately that $(\f a_2^i+\f a_2^j)^\star=A$. This proves
that $\ms F:=\{\f a_2^i:i\geq 1\}$ is a $\star$-comaximal collection
of $\star-$finite $\star-$ideals. Moreover, $\ms F$ is infinite,
since each element of $\ms  F$ is a proper $\star-$ideal. This gets
a contradiction.\end{proof}

\begin{prop}\label{caratterizzazione_star_carattere_finito}
Let $A$ be an integral domain and $\star$ a finite type star
operation on $A$. Then the following conditions are equivalent.

\begin{enumerate}[(i)]
\item $A$ has the $\star-$finite character.
\item For each nonzero element $a\in A$, every family $\gg \in (\Sigma')^\star_a$ is finite.
\end{enumerate}
\end{prop}

\begin{proof} (i) $\Rightarrow$ (ii) Fix a nonzero element $a\in A$ and let $\ms F$ be
an element of $(\Sigma')^\star_a$. For  every $\f a\in \ms F$, we
can pick a $\star-$maximal $\star-$ideal $\f m(\f a)$ containing $\f
a$ (and $a$, in particular). Since  $\ms F$ is a $\star$-comaximal
collection, for distinct ideals $\f a,\f b\in \ms F$, we have that
$\f m(\f a)\neq \f m(\f b)$. Thus the cardinality of $\ms F$ is less
 than or equal to  that of $\sMax(A)\cap V(a)$. Then (ii) follows
by assumption.

(ii) $\Rightarrow$ (i) Fix a nonzero element $a\in A$. By assumption
and Proposition \ref{fin_imply_add}, we can apply Proposition
\ref{condizionesufficientemax} (iii), and thus there exists a
 collection { of ideals $\ms F\in \Max_{\subseteq}((\Sigma')^\star_a)\cap \Sigma^\star_a$.} Then the conclusion follows immediately by Proposition
\ref{starcaratterefinitezza}.\end{proof}

%

Let $A$ be an integral domain and $\mathcal A$  a family of
overrings of $A$ such that $A = \bigcap\mathcal A$. For each $B\in
\mathcal A$, let $\star_B$ be a star operation on $B$.

Recall by \cite[Theorem 2]{a}, that $$I \mapsto I^{\wedge_{_{B\in
\mathcal A}} \star_B}: = \bigcap_{B\in\mathcal A} (IB)^{\star_B}$$
is a star operation on $A$. Moreover, if each $\star_B$ is of finite
type and the intersection $\bigcap\mathcal A$ is locally finite
(i.e., each element $d \in \bigcap\mathcal A$ is a unit in all
except a finite number of domains in $\mathcal A$), then
$\wedge_{_{B\in\mathcal A}} \star_B$ is of finite type.

\begin{prop} \label{locstarfin}
Let $A$ be an integral domain and $\Delta  \subseteq \sss(A)$.
Assume that $A = \bigcap_{\p \in \Delta} A_{\p}$ is a locally finite
intersection. For each $\p \in \Delta$, let $\star_\p$ be a finite
type star operation on $A_{\p}$  and set $\star := \wedge_{\p \in
\Delta} \star_\p$. If $\f a$ is an ideal of $A$ such that $\f
aA_{\p}$ is $\star_\p$-finite for each $\p$, then $\f a$ is
$\star-$finite.
\end{prop}

\begin{proof}
As usual, take an $x \in \f a$. Let $\p_1, \p_2, \ldots, \p_n$ the
only primes in $\Delta$ containing $x$. So, if $\p \neq \p_1,
\ldots, \p_n$, $\f aA_{\p} = xA_{\p} = A_{\p}$. Since the
$\star_\p$'s are of finite type, for each $i = 1, \ldots, n$, we can
find a finitely generated ideal  $\f a_i$ of $A$, included in $\f
a$, such that
 $(\f a A_{\p})^{\star_\p} = (\f a_i A_{\p})^{\star_\p}$. Let $\f b = (x) + \f a_1 + \ldots + \f a_n$.
Obviously $\f b$ is a finitely generated integral ideal of $A$
contained in $\f a$.

If $\p \neq \p_1 \ldots \p_n$, $\f bA_{\p} \supseteq xA_{\p} =
A_{\p}$ and so $(\f bA_{\p})^{\star_\p} = (A_{\p})^{\star_\p} = (\f
aA_{\p})^{\star_\p}$. Moreover, for $i = 1, 2, \ldots, n$, $(\f
aA_{\p_i})^{\star_{\p_i}} = (\f a_i A_{\p_i})^{\star_{\p_i}}
\subseteq (\f bA_{\p_i})^{\star_{\p_i}} \subseteq (\f a
A_{\p_i})^{\star_{\p_i}} $.

So, for each $\p \in \Delta$, $(\f aA_{\p})^{\star_\p} = (\f b
A_{\p})^{\star_\p}$. It follows that $\f a^\star = \f b^\star$ and
$\f a$ is $\star-$finite.
\end{proof}

We will denote by $t_\p$ the $t$-operation of $A_{\p}$. Recall  that
B.G. Kang proved in  \cite[Lemma 3.4(3)]{k}  that, for each ideal
$\f a$, $(\f aA_{\p})^{t_\p} = (\f a^tA_{\p})^{t_\p}$.

\begin{prop}\label{twedge}
Let $A$ be an integral domain, $\Delta \subseteq \sss(A)$ such that
$A = \bigcap_{\p \in \Delta}A_{\p}$ is a locally finite
intersection. Then $t= \wedge_{\p \in \Delta} t_\p$.
\end{prop}

\begin{proof}
Since the intersection $A= \bigcap_{\p \in \Delta} A_{\p}$ is
locally finite and $t_\p$ is of finite type for each $\p\in\Delta$,
by \cite[Theorem 2]{a} it follows that $\star := \wedge_{\p \in
\Delta} t_\p$ is a finite type star operation on $A$. Thus $\star
\leq t$. Conversely, for an ideal $\f a$, we have that $\f a^t
\subseteq \f a^t A_{\p}$ for each $\p \in \Delta$, so $\f a^t
\subseteq \bigcap_\p \f a^t A_{\p} \subseteq \bigcap_\p (\f a^t
A_{\p})^{t_\p} =  \bigcap_\p (\f a A_{\p})^{t_\p} = \f a^\star$ (the
first equality follows by Kang's result). Thus $t \leq \star$. Hence
$t = \star$.
\end{proof}

\begin{thm}\label{t-locally t-finite}
Let $A$ be an integral domain with the $t$-finite character. Then
each $t$-locally $t$-finite $t$-ideal is $t$-finite.
\end{thm}
\begin{proof}
It follows immediately from Proposition \ref{locstarfin} and
Proposition \ref{twedge}, taking $\Delta = \tmax(A)$.
\end{proof}
The following technical result will be crucial, in the following.
\begin{lem}\label{claim} Let $A$ be an integral domain, $\star$ be a finite type star operation
on $A$, $\ms F$ a collection of ideals of $A$ such that $\bigcap \ms
F$ contains a nonzero element $a\in A$. Set $$\f a :=\left\{ x\in
A:x\prod \ms F'\subseteq aA,\mbox{ for some finite subset } \ms
F'\subseteq \ms F\right\}.
$$
 Then the following conditions hold:
 \begin{enumerate}[\rm (a)]

\item let $F:=\{x_1,\z,x_n\}$ be a nonempty  finite subset of $\f a$. For each $i=1,\z,n$, let $\ms F_{x_i}$ be a finite subcollection of $\ms F$
such that $x_i\prod \ms F_{x_i}\subseteq (a)$, and let $y\in
(x_1,\z,x_n)^\star$, $\ms F':=\bigcup_{i=1}^n\ms F_{x_i}$. Then
$y\prod \ms F'\subseteq (a)$ and, in particular, we have
$$(x_1,\z,x_n)^\star\subseteq \f a.$$
\item  $\f a$ is
a $\star-$ideal.
\item If $\ms F$ is a $\star$-comaximal collection over $a$ of
ideals of $A$, then for each $\f m \in \star-\Max(A)$, we have that:

$$\f aA_{\f m} =
\left\{\begin{array}{ll}
aA_{\f m} &\mbox{ if }  \f m\nsupseteq \f b, \mbox{ for each } \f b \in \ms F \\
a(A \colon \f b_0)A_{\f m} &\mbox{ if } \f b_0 \mbox{ is the unique
ideal in } \ms F \mbox{ contained in } \f m
\end{array}
\right.$$
\item If $\f a$ is $\star-$finite, then there exists a finite subcollection $\widehat{\ms G}$ of $\ms F$ such that
$$
\f a=\left\{x\in A: x\prod \widehat{\ms G}\subseteq aA\right\}.
$$
\end{enumerate}
\end{lem}

\begin{proof}
Firstly, we note that $\f a$ is an ideal of $A$, as it is
immediately seen.

(a). As a matter of fact, we have
$$
y\prod \ms F'\subseteq \left((x_1,\z,x_n)^\star\prod \ms
F'\right)^\star\subseteq \left(\sum\{x_i\prod\ms
F_{x_i}:i=1,\z,n\}\right)^\star\subseteq (a).
$$

(b). Let $x\in \f a^\star$. Since $\star$ is a star operation of
finite type, then there exist elements $f_1,\z,f_n\in \f a$ such
that $x\in (f_1,\z,f_n)^\star$. Then by statement (a), $(f_1,\z
,f_n)^\star\subseteq \f a$. Thus $\f a$ is a $\star-$ideal.

(c) Let $\f m$ be a $\star-$maximal ideal of $A$. Since $\ms F$ is a
$\star$-comaximal collection, $\ms F_{\f m}:=\{\f b \in \ms F:\f
b\subseteq \f m\}$ has cardinality at most one. For each $\f b\in
\ms F\setminus \ms F_{\f m}$, we can pick an element $x_{\f b}\in \f
b\setminus \f m$. If $\ms F_{\f m}$ is empty it is clear that $\f
aA_{\f m}=aA_{\f m}$. In fact, let $\dfrac{x}{s}\in \f aA_{\f m}$,
$x\in\f a$ and $s\in A\setminus \f m$. By definition, there exists a
finite subcollection  $\ms F'\subseteq \ms F$ such that $x\prod \ms
F'\subseteq aA$. Then $\prod_{\f b\in \ms F'}x_{\f b} \in A\setminus
\f m$ and thus
$$
\frac{x}{s}=\frac{x\prod_{\f b\in \ms F'}x_{\f b}}{s\prod_{\f b\in
\ms F'}x_{\f b}}= \frac{ab}{s\prod_{\f b\in \ms F'}x_{\f b}}\in
aA_{\f m}
$$
for some $b\in A$. Now, suppose that there is a unique ideal $\f
b_0\in \ms F_{\f m}$. It is clear that $a(A:\f b_0)\subseteq \f a$,
and hence $a(A:\f b_0)A_{\f m}\subseteq \f a A_{\f m}$. Conversely,
let $\dfrac{x}{s}\in \f aA_{\f m}$ $x\in \f a,s\in A\setminus \f m$
and let $\ms F'$ be a finite subcollection of $\ms F$ such that
$x\prod \ms F'\subseteq aA$. If $\f b_0\notin \ms F'$, then the same
argument given above shows that $\dfrac{x}{s}\in aA_{\f m}\subseteq
a(A:\f b_0)A_{\f m}$. If $\f b_0\in \ms F_{\f m}$, set $p:=\prod_{\f
b\in \ms F'\setminus\{\f b_0\}}x_{\f b}$, if $\ms F'\setminus\{\f
b_0\}\neq \emptyset$, $p:=1$ otherwise. Then  $xp\f b_0\subseteq aA$
and $p\notin \f m$. It follows that   $xpA \subseteq a(A:\f b_0)$
  keeping in mind that $(A:\f b_0)$ is a always a fractional
$\star-$ideal. Then we have
$$
\frac{x}{s}=\frac{xp}{sp}\in a(A:\f b_0)A_{\f m},
$$
and thus (c) is proved.

(d).  Pick elements $x_1,\z,x_r\in \f a$ such that $\f a=\f
a^*=(x_1,\z,x_r)^\star$. For each $i\in \{1,\z,r\}$, let  $\ms G_i$
be  a finite subcollection  of $\ms F$ such that $x_i\prod\ms
G_i\subseteq (a)$. Thus it is enough to choose $\widehat{\ms
G}:=\bigcup_{i=1}^r\ms G_i$ and apply statement (a) to get the
equality $
 \f a=\{x\in A:x\prod \widehat{\ms G}\subseteq (a)\}.
$
 \end{proof}

 We recall that a domain $A$ is
\textit{$v$-coherent}  if for any nonzero finitely generated ideal
$I$ of $A$, $I^{-1}$ is $v$-finite that is, $I^{-1} = J^v$ for some
$J \in \bbf(A)$ (\cite[Proposition 3.6]{FG}). Important classes of
$v$-coherent domains are Noetherian domains, Mori domains, Pr\"ufer
domains, P$v$MD's, finite conductor domains (i.e., $(x)  \cap (y)$
is finitely generated for each $x,y \in A$), coherent domains (i.e.,
the intersection of two finitely generated ideals is finitely
generated). A domain $A$ is \textit{$t$-locally $v$-coherent} if
$A_{\f m}$ is $v$-coherent , for each   $\f m \in \tmax(A)$.

\begin{thm}\label{thm:characterization t-finite}
Let $A$ be an integral domain which is $t$-locally $v$-coherent.
Then the following conditions are equivalent.

\begin{enumerate}[(i)]
\item $A$ has the $t$-finite character;

\item every family of $t$-finite, $t$-comaximal, $t$-ideals over a nonzero element $a \in A$ is finite;

\item every nonzero $t$-locally $t$-finite ideal $I$ (i.e., for any $\m \in \tmax(A)$, $I_\m$
is $t$-finite with respect to the $t$-operation of $A_{\m}$) is
$t$-finite.
\end{enumerate}
\end{thm}

\begin{proof}
The equivalence (i) $\Leftrightarrow$ (ii) is given by Proposition
\ref{caratterizzazione_star_carattere_finito} and it
 holds more in general for any domain $A$ and by replacing $t$ with any finite type star operation.

(iii) $\Rightarrow$ (ii).

Let $\ms F \in \Sigma^t_a$ (where, we recall that $\Sigma^t_a$ is
the set of collections of $t$-comaximal, $t$-finite $t$-ideals of
$A$ containing $a$) and set
$$\f a:=\left\{ x\in A:x\prod \ms F'\subseteq (a),\mbox{ for some finite subset } \ms F'\subseteq \ms F\right\}.
$$
By Lemma \ref{claim}, $\f a$ is a $t-$ideal of $A$ and, moreover,
for each $t-$maximal ideal $\f m$ of $A$, we have
$$\f aA_{\f m} =
\left\{\begin{array}{ll}
aA_{\f m} &\mbox{ if }  \f m\nsupseteq \f b, \mbox{ for each } \f b \in \ms F \\
a(A \colon \f b_0)A_{\f m} &\mbox{ if } \f b_0 \mbox{ is the unique
ideal in } \ms F \mbox{ contained in } \f m
\end{array}
\right.$$ Thus, if $\f b_0$ is the unique ideal of $\ms F$ contained
in $\f m$, we have that $\f aA_{\f m}=a(A:\f b_0)A_{\f m} = (A_{\f
m}:\f b_0A_{\f m})$, and it is $t_{\f m}$-finite since $\f b_0A_{\f
m}$ is $t_{\f m}$-finite (see  \cite[Lemma 3.4(3)]{k})  and $A_{\f
m}$ is $v$-coherent. This shows that $\f a$ is a $t$-locally
$t$-finite, $t$-ideal and hence it is $t$-finite by assumption.
Again by Lemma \ref{claim}, it follows that
$$\f a=\left\{x\in A:x\prod \widehat{\ms G}\subseteq (a)\right\},$$
for some finite subcollection $\widehat{\ms G}$ of $\ms F$. Now,
consider an ideal $\f b_1\in \ms F\setminus\widehat{\ms G}$. By the
definition of $\f a$, it follows immediately that $a(A:\f
b_1)\subseteq \f a$, and thus $a(A:\f b_1)\prod \widehat{\ms
G}\subseteq aA$, that is, $(A:\f b_1)\prod \widehat{\ms G}\subseteq
A$. Hence we have that $\prod \widehat{\ms G} \subseteq  (A:(A:\f
b_1))= \f b_1^{v}=\f b_1$, keeping in mind that $\f b_1$ is a
$t$-finite $t$-ideal, so it is divisorial. Since $\ms F$ is a
$t$-comaximal collection of ideals and $\f b_1\in \ms F\setminus
\widehat{\ms G}$, we have that
$$
A=A^t=\left(\prod_{\f b\in \widehat{\ms G}}(\f b_1+\f
b)^t\right)^t=\left(\prod_{\f b\in \widehat{\ms G}}(\f b_1+\f
b)\right)^t=\left(\prod\widehat{\ms G}+\f c\right)^t,
$$
for some ideal $\f c\subseteq \f b_1$. Finally, using the fact that
$\prod \widehat{\ms G}\subseteq \f b_1$, it follows that
$$A\subseteq \left(\prod\widehat{\ms G}+\f c\right)^t\subseteq \f b_1^t =\f b_1.
$$
This proves that $\{A\}=\ms F\setminus \widehat{\ms G}$. Since
$\widehat{\ms G}$ is finite, the statement follows.

(i) $\Rightarrow$ (iii) See Theorem \ref{t-locally t-finite}, and
observe that it holds for any domain (without the $t$-locally
$v$-coherence hypothesis) .
\end{proof}

\begin{remarks}

\begin{enumerate}[(a)]

\item A $t$-locally $v$-coherent domain is not necessarily a P$v$MD.
In fact any  Noetherian domain is $t$-locally $v$-coherent (and it
is not always P$v$MD). This shows that Theorem
\ref{thm:characterization t-finite} strictly generalizes the results
obtained for Pr\"ufer domains and P$v$MD's (\cite[Proposition
6.11]{hk} and \cite[Proposition 5]{zaf}).

\item  More generally a $t$-locally $v$-coherent domain is not necessarily $v$-coherent (while a $v$-coherent domain is locally $v$-coherent). The example of a locally $v$-coherent domain that is not $v$-coherent given in \cite[Example 3.3]{gh} is easily seen to be also $t$-locally $v$-coherent.
However, as a consequence of Proposition \ref{t-locally t-finite}, a
$t$-locally $v$-coherent domain with $t$-finite character is
$v$-coherent (cf. with \cite[Proposition 3.1]{gh}, where it is
mentioned that locally $v$-coherent domains with finite character
are $v$-coherent). So, the domains of Theorem
\ref{thm:characterization t-finite} are in fact $v$-coherent.

\item Note that by Proposition \ref{locstarfin} and Proposition \ref{twedge}
it follows easily that if $A$ can be written as a locally finite
intersection of $v$-coherent localizations, then $A$ is
$v$-coherent. This generalizes both the results in (b) and
\cite[Proposition 3.1]{gh}.

\item The main hypothesis of Theorem
\ref{thm:characterization t-finite} stating that $A$ is $t$-locally
$v$-coherent is not always required in order to get the equivalences
(i)-(ii)-(iii). In fact, take $A$ to be a pseudo-valuation domain
with maximal ideal $M$ and such that $M$ is not principal as an
ideal of $M^{-1}=V$. This is equivalent to requiring that $A$ is not
$v$-coherent (\cite[Remark 3.14]{gh}). It is well-known that $M$ is
a $t$-ideal of $A$, so $A$ is not a $t$-locally $v$-coherent domain,
but conditions (i)-(ii)-(iii) are trivially satisfied.

\item In general, Theorem \ref{thm:characterization t-finite} cannot
be extended to any finite type star operation. For instance, if we
take the identity operation $d$ at the place of $t$, Theorem
\ref{thm:characterization t-finite} fails. In fact, a Noetherian
domain do not need to have the finite character on maximal ideals,
but each locally finitely generated ideal is finitely generated.
\end{enumerate}
\end{remarks}

We have already observed that Noetherian domains may not have the
finite character, but they always have the $t$-finite character
(since, more generally, Mori domains have the $t$-finite character
\cite[Proposition 2.2]{BG}). Conversely, the following example shows
that there exist domains with the finite character which do not have
the $t$-finite character.

\begin{ex}\label{ex:finite character}
We consider the following pullback diagram:

$$
\CD
R  @>>> D\\
@VVV @VVV \\
V @>{\varphi}>> k = V/M
\endCD
$$

\smallskip
\noindent where $V$ is a valuation domain with maximal ideal $M$,
residue field $k$ and $D$ is a $2$-dimensional, local Krull domain
with quotient field $k$.

The domain $R$ turns out to be local. By \cite[Theorem 2.18]{gh} the
$t$-maximal ideals of $R$ are the inverse images of the $t$-maximal
ideals of $D$ (which are the height-one primes) and they all contain
$M$. Thus $R$ does not have the $t$-finite character. But $R$  has
the finite character as being local.
\end{ex}

\section{The local character of $\star-$invertibility}
The following result has been proved by M. Zafrullah
\cite[Proposition 4]{zaf} for the $t$-operation. We give a different
proof of Zafrullah's result and generalize it to every star
operation of finite type. The proof of the following Proposition
uses an argument similar to the one given in  the proof of Theorem
\ref{thm:characterization t-finite}.

\begin{prop} \label{starlocprinc}
Let $A$ be an integral domain, $\star$ be a star operation of finite
type on $A$ and $a\in A\setminus\{0\}$. Assume that every
$\star-$locally principal ideal of $A$ is $\star-$invertible. Then
every $\star$-comaximal collection over $a$ of $\star-$invertible
$\star-$ideals is finite.
\end{prop}

\begin{proof} Let $\ms F$ be a $\star$-comaximal collection over $a$ of
$\star-$invertible $\star-$ideals. Now, let
$$\f a:=\left\{ x\in A:x\prod \ms F'\subseteq (a),\mbox{ for some finite subset } \ms F'\subseteq \ms F\right\}.
$$
By Lemma \ref{claim} and the fact that every ideal in $\ms F$ is
$\star-$invertible, it follows that $\f a$ is a $\star-$locally
principal $\star-$ideal, and hence it is $\star-$finite, since it is
$\star-$invertible, by assumption. Thus, by Lemma \ref{claim}(c),
there exists a finite subcollection $\widehat{\ms G}$ of $\ms F$
such that
$$
\f a=\left\{x\in A: x\prod \widehat{\ms G}\subseteq (a)\right\}.
$$
 Now, consider an ideal $\f b_1\in \ms F\setminus\widehat{\ms G}$.
By the  definition of $\f a$, it follows immediately that $a(A:\f
b_1)\subseteq \f a$, and thus $a(A:\f b_1)\prod \widehat{\ms
G}\subseteq aA$, that is, $(A:\f b_1)\prod \widehat{\ms G}\subseteq
A$. Hence we have
$$
\prod \widehat{\ms G}=\prod \widehat{\ms G} ((A:\f b_1)\f
b_1)^\star\subseteq \left(\prod \widehat{\ms G}(A:\f b_1)\f
b_1\right)^\star\subseteq \f b_1^{\star}=\f b_1,
$$
keeping in mind that $\f b_1$ is a $\star-$invertibile
$\star-$ideal. Since $\ms F$ is a $\star$-comaximal collection of
ideals and $\f b_1\in \ms F\setminus \widehat{\ms G}$, we have
$$
A=A^\star=\left(\prod_{\f b\in \widehat{\ms G}}(\f b_1+\f
b)^\star\right)^\star=\left(\prod_{\f b\in \widehat{\ms G}}(\f
b_1+\f b)\right)^\star=\left(\prod\widehat{\ms G}+\f c\right)^\star,
$$
for some ideal $\f c\subseteq \f b_1$. Finally, using the fact that
$\prod \widehat{\ms G}\subseteq \f b_1$, we have
$$A\subseteq \left(\prod\widehat{\ms G}+\f c\right)^\star\subseteq \f b_1^\star =\f b_1.
$$
This proves that $\{A\}=\ms F\setminus \widehat{\ms G}$. Since
$\widehat{\ms G}$ is finite, the statement follows. \end{proof}

\bigskip

\begin{prop}\label{t-loc star-loc}
Let $A$ be a domain in which   each   $t$-locally principal
$t$-ideal is $t$-finite. Then, for any star operation of finite
type, each $\star-$locally principal $\star-$ideal is
$\star-$finite.   In particular, a locally principal ideal is
finitely generated (and so, invertible).
\end{prop}

\begin{proof}
Let $I$ be a $\star-$ideal of $A$ that is $\star-$locally principal.
Then $I$ is also $t$-locally principal since a $t$-ideal is in
particular a $\star-$ideal and each $t$-maximal ideal is contained
in a $\star-$maximal ideal. So $I$ is $t$-finite, say $I := (a_1,
\cdots,a_n)^t$, for some $a_1, \cdots, a_n \in  I $. We show that
$I$ is $\star-$invertible, that is $(II^{-1})^\star = A$, and so it
is $\star-$finite. To see this, it is enough to show that for each
$\star-$maximal ideal $\f m$ of $A$, $(II^{-1})_{\f m} = A_{\f m}$.
Now, if $IA_{\f m} = aA_{\f m}$, we have that:
\begin{align}
\nonumber A_{\f m} \supseteq (II^{-1})_{\f m} &= I_{\f m}I^{-1}_{\f
m} =
  aA_{\f m}(a_1^{-1}A \cap \cdots \cap a_n^{-1}A)_{\f m} \\
\nonumber &= aa_1^{-1}A_{\f m} \cap \cdots \cap aa_n^{-1}A_{\f m}
\supseteq A_{\f m}.\end{align}

 This proves the statement.
\end{proof}

Proposition \ref{t-loc star-loc} does not hold if we replace the
$t$-operation with another finite type star operation $\star$. This
is what we show, for instance, in the next example with $\star = d$.

\begin{ex}\label{inv not tinv}

Consider the domain $R$ constructed in Example \ref{ex:finite
character}. Since $R$ is local, each locally principal ideal is
invertible.

By \cite[Theorem 4.13]{gh} $R$ is a P$v$MD, so it  is $t$-locally
$v$-coherent.  Since $R$ does not have the $t$-finite character,
then
  there exists a nonzero
ideal $I$ that is $t$-locally principal but not $t$-finite.

\end{ex}

\bigskip
Now we  look more closely to the proof of Proposition
\ref{starlocprinc} and show that this result can be slightly
improved in the case of the $t$-operation. We will call an ideal $\f
a$   \textit{$t$-principal} if $\f a^t$ is principal.

 The fact that
the family $\ff$ is composed of $\star-$invertible ideals is used
twice in the proof. First, to show that $a A_\m = a (A : \f
b_0)A_\m$ is a principal ideal. Here, it is clear that it is
sufficient to assume that $(A : \f b_0)$ is principal when localized
at the $\star-$maximal ideals. 
If we assume $\b_0$  to be $t$-finite, it is easily seen that
$(A:\mathfrak{b}_0)A_\m = (A_\m : \mathfrak{b}_0 A_\m)$. So
$(A:\mathfrak{b}_0)A_\m$ is principal if and only if $(\f b_0A_\m)^t
=(\f b_0 A_\m)^v  $ is principal. Thus, in this point of the proof,
we use only the fact that $\f b_0$ is $t$-finite and $\f b_0 A_\m$
is $t$-principal.

The other point of the proof where   $\star-$invertibility for
ideals   of $\ff$ is used is where   from the containment $(A:\f
b_1)\prod \widehat{\ms G}\subseteq A$ we deduce that $\prod
\widehat{\ms G}\subseteq \f b_1$. But this is true in our case since
$\f b_1$ is a $t$-finite $t$-ideal.

In conclusion, for the $t$-operation, Proposition \ref{starlocprinc}
  can be   modified  as follows:

\begin{prop} \label{tlocprinc}
Let $A$ be an integral domain and $a\in A\setminus\{0\}$. Assume
that every $t$-locally principal ideal of $A$ is $t$-invertible.
Then every $t$-comaximal collection over $a$ of  $t$-finite
$t$-locally $t$-principal $t$-ideals is finite.
\end{prop}

We give an application of this result,  to obtain a further
generalization of Bazzoni's conjecture.

Recall  that if $(x_1, \ldots, x_n)_v = (x)$ for some $x$ in $A$,
then $x$ is called the $v$-gcd of $x_1, \ldots, x_n$ and all pairs
of elements of $A$ have a $v$-gcd if and only if $D$ is a
GCD-domain.

\begin{prop}\label{gcd} For a domain $A$ the following conditions
hold:

\begin{enumerate}
\item  if $A$ is a GCD domain then it is $v$-coherent.

\item  if $A$ is a $t$-locally GCD domain, then a nonzero $t$-finite $t$-ideal of $A$ is $t$-locally $t$-principal.

\end{enumerate}
\end{prop}

\begin{proof} \begin{enumerate}
\item  It follows
from the fact that a GCD domain is a P$v$MD, and P$v$MD's are
$v$-coherent

\item Let $\f a$ be a $t$-finite $t$-ideal and $\m$ a $t$-maximal
ideal. Let $\f b$ be a finitely generated ideal of $D$ such that $\f
b^t = \f a$ and let $x$ be a $v$-GCD of the generators of $\f
bA_\m$. We have $$(\f aA_\m)^{t_\m} = (\f b^tA_\m)^{t_\m} = (\f
bA_\m)^{t_\m} = (xA_\m)^{t_\m}.$$

Thus, $\f a$ is $t$-locally $t$-principal.

\end{enumerate}
\end{proof}

\begin{prop} \label{locgcd}
Let $A$ be a $t$-locally GCD domain. If every $t$-locally principal
ideal of $D$ is $t$-invertible then $D$ has the $t$-finite
character.
\end{prop}

\begin{proof}
Note that by Proposition \ref{gcd}(2), $A$ is $t$-locally
$v$-coherent.  Let $a$ be an element of $A$. Take a $t$-comaximal
collection $\mathcal F$ over $a$ of proper $t$-finite $t$-ideals. By
Proposition \ref{gcd}, $\mathcal{F}$ is in particular a
$t$-comaximal collection over $a$ of $t$-finite $t$-locally
$t$-principal $t$-ideals. So, it is finite by Proposition
\ref{tlocprinc}. Hence $D$ has the $t$-finite character by Theorem
\ref{thm:characterization t-finite}.
\end{proof}

Proposition \ref{locgcd} extends the $t$-version of Bazzoni's
conjecture (\cite[Proposition 6.11]{hk} and \cite[Proposition
5]{zaf}) to $t$-locally GCD domains. Note that locally GCD domains
are not necessarily P$v$MD (\cite[Example 2.6]{zaf2}). Now, if $\p$
is a $t$-maximal ideal of a locally GCD domain $A$, there exists a
maximal ideal $\m$ of $A$ such that $\p \subseteq \m$. So $A_\p$ is
a localization of the GCD domain $A_\m$, and so it is a GCD domain.
Thus a locally GCD domain is also $t$-locally GCD and so the class
of $t$-locally GCDs is properly larger than the class of P$v$MDs.

\smallskip
However, we can  prove that a $t$-locally GCD domain $A$ such that
every $t$-locally principal ideal of $A$ is $t$-invertible is a
P$v$MD. More precisely, we will show that a $t$-locally GCD domain
with the $t$-finite character is a P$v$MD.

%
%
%
%

First we notice that, slightly modifying   the proof of Proposition
\ref{starlocprinc}, similarly to as we have done to obtain
Proposition \ref{tlocprinc}, we have the following result.

\begin{prop} \label{t-loc t-inv}
Let $A$ be an integral domain and $a\in A\setminus\{0\}$. Assume
that every $t$-locally $t$-invertible $t$-ideal of $A$ is
$t$-invertible. Then every $t$-comaximal collection of   $t$-locally
$t$-invertible $t$-ideals over $a$ is finite.
\end{prop}

\begin{prop}\label{locpvmd}
Let $A$ be a $t$-locally P$v$MD. The following are equivalent:

\begin{enumerate}
\item[(i)] Each $t$-locally $t$-invertible $t$-ideal is $t$-invertible.

\item[(ii)] $A$ has the $t$-finite character.

\item[(iii)] $A$ is a P$v$MD with the $t$-finite character.
\end{enumerate}
\end{prop}

\begin{proof}
(iii) $\Rightarrow$ (ii) It is trivial.

(ii) $\Rightarrow$ (i) This follows from Theorem \ref{t-locally
t-finite} for any integral domain $A$.

(i) $\Rightarrow$ (iii) Let $I$ be a $t$-finite $t$-ideal of $A$.
Then $I$ is $t$-locally finite, so it is $t$-locally $t$-invertible
(since $A$ is a $t$-locally P$v$MD) and, by assumption (i), $I$ is
$t$-invertible. This shows that $A$ is a P$v$MD.

Now, let $\ms F$ be a $t$-comaximal collection of $t$-finite
$t$-ideals (over a nonzero element $a \in A$).
 Then, by the fact that $A$ is a $t$-locally P$v$MD, $\ms F$
is also a $t$-comaximal family of $t$-locally $t$-invertible
$t$-ideals. By  assumption (i) and   by Proposition \ref{t-loc
t-inv}, it follows that $\ms F$ is finite. Since a P$v$MD is
$v$-coherent,   $A$ has the $t$-finite character by Theorem
\ref{thm:characterization t-finite}.
\end{proof}

We easily have the following corollary.

\begin{cor}\label{locgcd2}
Let $A$ be a $t$-locally GCD domain. The following are equivalent:

\begin{enumerate}
\item[(i)] Each $t$-locally principal $t$-ideal is $t$-invertible.
\item[(ii)] $A$ has the $t$-finite character.
\item[(iii)] $A$ is a P$v$MD with the $t$-finite character.
\end{enumerate}
\end{cor}



\bigskip

\begin{thebibliography}{12}

\bibitem{a} {\sc D.D. Anderson},  {\em Star-operations induced by overrings},  Comm. Algebra \textbf{16} (1988), no. 12, 2535--2553.

\bibitem{aa2} {\sc D.D. Anderson and D.F. Anderson}, {\em Some remarks on cancellation ideals. Math.
Japon},  {\bf 29} no. 6 (1984),  879--886.

\bibitem{BG} {\sc V. Barucci and S. Gabelli}, \emph{How far is a Mori domain from
being a Krull domain?}, J. Pure Appl. Algebra {\bf 45} (1987), no.
2, 101--112.

\bibitem{bazzoni} {\sc S. Bazzoni}, {\em Class semigroups of Pr\"ufer domains}, J. Algebra
{\bf 184} (1996),  613–-631.


\bibitem{FG} {\sc M. Fontana and S. Gabelli}, {\em On the class group and the local
class group of a pullback}, J. Algebra {\bf 181} (1996), no. 3,
803--835

\bibitem{gh} {\sc S. Gabelli and E.G. Houston}, {\em Coherent-like conditions in pullbacks},
Michigan Math. J. \textbf{44} (1997), no. 1, 99--123.

\bibitem{GV}{\sc S. Glaz and W.V. Vasconcelos},
{\em Flat ideals. II}, Manuscripta Math. {\bf 22} no.4 (1977),
325--341.



\bibitem{flat3}{\sc S. Glaz and W.V. Vasconcelos},
{\em Flat ideals. III}, Comm. Algebra {\bf 12} no.2 (1984),
199--227.


\bibitem{hk} {\sc F. Halter-Koch}, {\em Clifford semigroups of ideals in monoids and domains}, Forum Math., to appear.

\bibitem{hmmt} {\sc W.C. Holland, J. Martinez, W.Wm. McGovern and M.Tesemma}, {\em Bazzoni's Conjecture}, J. Algebra \textbf{320} (2008), 1764-1768.

\bibitem{k} {\sc B. G. Kang},  {\em Pr\"ufer $v$-multiplication domains and the ring $R[X]\sb {N\sb v}$}, J. Algebra \textbf{123} (1989), no. 1, 151--170.

\bibitem{mc} {\sc W. Wm. McGovern}, {\em Prufer domains with Clifford Class semigroup}, J. Comm. Alg., to appear.

\bibitem{PT} {\sc G.Picozza and F.Tartarone}, {\em Flat ideals and stability  in integral domains}, preprint.

\bibitem{zaf2} {\sc M. Zafrullah},  {\em The $D+XD\sb S[X]$ construction from GCD-domains}, J. Pure Appl. Algebra \textbf{50} (1988), no. 1, 93--107.



\bibitem{zaf} {\sc Muhammad Zafrullah}, {\em t-Invertibility and Bazzoni-like statements},
 J. Pure Appl. Algebra, to appear.

\end{thebibliography}
\end{document}